\documentclass[12pt]{article}
\usepackage{graphicx}
\usepackage{latexsym}
\usepackage{amsmath}
\usepackage{verbatim}
\usepackage{amssymb}
\usepackage{mathrsfs}
\usepackage{amsmath,amssymb}
\DeclareMathAlphabet      {\mathbf}{OT1}{cmr}{bx}{n}
\DeclareFontFamily{OT1}{pzc}{}
\DeclareFontShape{OT1}{pzc}{m}{it}%
{<-> s * [1.15] pzcmi7t}{}
\DeclareMathAlphabet{\mathpzc}{OT1}{pzc}{m}{it}

\usepackage{ dsfont }
\usepackage{ mathrsfs }

\usepackage{amsbsy }
\usepackage{graphics}

\usepackage{amsthm}
\newtheorem{thm}{Theorem}
\newtheorem{theorem}{Theorem}[section]

\newtheorem{lemma}[theorem]{Lemma}
\newtheorem{corollary}[theorem]{Corollary}

\newtheorem{definition}[theorem]{Definition}

\newtheorem{remark}{Remark}

\usepackage{romannum}
\usepackage[toc,page,title,titletoc,header]{appendix}
\usepackage{hyperref}

\usepackage[all,cmtip]{xy}

\def\ack{\section*{Acknowledgements}%
  \addtocontents{toc}{\protect\vspace{6pt}}%
  \addcontentsline{toc}{section}{Acknowledgements}%
}

\usepackage{abstract}

\begin{document} \small
\pagenumbering{arabic}
\title{Stability of the  braid  types   defined by  the symplecticmorphisms preserving a link}
\author{{ Guanheng Chen}}
\date{}
\maketitle
\thispagestyle{empty}
\begin{abstract}    
Fix a suitable link on the  disk.  Recently, F. Morabito associates each Hamiltonian symplecticmorphism preserving the link to a braid type. Based on  this construction, Morabito defines   a family of  pseudometrics  on the braid groups by using the Hofer metric. In this paper, we show that two Hamiltonian symplecticmorphisms define  the  same braid type  provided that their Hofer distance is  sufficiently small. As a corollary, the  pseudometrics   defined by Morabito  are nondegenerate.
\end{abstract}

\section{Introduction and main results}
\paragraph{Preliminaries}

Let $\widetilde{Conf^k}(\mathbb{D}): =\{(x_1,..,x_k) \in \mathbb{D}^k: x_i \ne x_j \mbox{ if } i \ne j \}$, where $\mathbb{D}$ is the disk.  The configuration space is $Conf^k(\mathbb{D}): =\widetilde{Conf^k}(\mathbb{D}) /S_k $, where $S_k$ is the permutation group.  Then the \textbf{braid group} is $\mathcal{B}_k := \pi_1(Conf^k(\mathbb{D}))$.

To see the elements   in $\mathcal{B}_k $ more geometrically,  we fix a base point $\mathbf{x} = (x_1, ...x_k) \in \widetilde{Conf^k}(\mathbb{D})$. Then  a \textbf{braid} is a collection of paths $\{\gamma_i: [0,1] \to \mathbb{D} \}_{i=1}^k$ such that
\begin{enumerate}
\item
$\gamma_i(0)= x_i$
and $\gamma_i(1) = x_{\sigma(i)}$, where $\sigma \in S_k$.
\item
$\gamma_i(t) \ne \gamma_j(t)$ for $i \ne j$.
\end{enumerate}
Obviously, this gives   a loop $\gamma: S^1 \to Conf^k(\mathbb{D})$.  Conversely, we can construct a braid from a loop in  $Conf^k(\mathbb{D})$ easily. So we don't distinguish these two concepts throughout  this note.
The homotopy class $[\gamma] \in \mathcal{B}_k $ is called a  \textbf{braid type}. For more details about the braid group, we refer the readers to \cite{KT}.


The braid group is a classical object in low dimensional topology and it has been studied by mathematicians  for decades.    Recently,   some progress  has been   made to study  the braid group from  view points of  symplectic geometry \cite{AM, FM, MK, H3}.   To describe their results, we first review some definitions  in symplectic geometry quickly.

Let $\Sigma$ be  a surface  with a volume form $\omega$.  
 A function $H: [0,1]_t \times \Sigma \to \mathbb{R}$ is called  \textbf{a Hamiltonian function}.  We can associate it to a vector field $X_H$ called the \textbf{Hamiltonian vector field} by the relation $\omega(X_H, \cdot) =d_{\Sigma }H$. Let $\varphi_H^t$ be the flow generated by $X_H$. A diffeomorphism $\varphi$ of $\Sigma $ is called a \textbf{Hamiltonian symplecticmorphism} if $\varphi =\varphi^1_H$ for some $H$.  Let $Ham(\Sigma, \omega)$ be the set of all Hamiltonian symplecticmorphisms. In fact, $Ham(\Sigma, \omega)$ is a group.

There is a distance function  on $Ham(\Sigma, \omega)$ called the \textbf{Hofer metric} \cite{HH}.  It is defined as follows.   Given a Hamiltonian function $H$, the Hofer norm of $H$ is $$|H|_{(1, \infty)} : =\int_0^1 (\max_{\Sigma} H_t - \min_{\Sigma} H_t)  dt.$$
The \textbf{Hofer norm} of a Hamiltonian symplecticmorphism is defined by
$$|\varphi|_{Hofer} : = \inf \{ |H|_{(1, \infty)}  : \varphi =  \varphi^1_H   \}.$$
Given $\varphi_1, \varphi_2 \in Ham(\Sigma, \omega)$, the Hofer metric is $d_{Hofer}(\varphi_1, \varphi_2) :=|\varphi_1 \circ \varphi_2 ^{-1}|_{Hofer}. $

\paragraph{Main results}
Given a union of distinct 1-periodic orbits $ \alpha$ of $\varphi_H^1$, one  can view them as a braid in $\Sigma$ and hence we get a braid type  $b(\alpha)$ associated with  these periodic orbits.  M.R.R. Alves and M. Meiwes show that  if another symplecticmorphism $\varphi^1_{H'}$ is sufficiently close to $\varphi_H^1$ with respect to the Hofer metric, then  there is a union of  distinct 1-periodic orbits $\alpha'$ of $\varphi_{H'}^1$ such that $b(\alpha) =b(\alpha')$ \cite{AM}.   Recently, M. Hutchings generalizes  this result to the general cases  (without the Hamiltonian assumption and $\alpha$ can contain periodic orbits with arbitrary periods) by using   holomorphic curve methods \cite{H3}.

In \cite{FM}, F. Morabito constructs braid types  for the Hamiltonian symplecticmorphisms  preserving a fixed link.  The braid types   come from the Hamiltonian chords rather than periodic orbits. In this note, we follow  Alves and  Meiwes's idea to prove a parallel result  for   the braid types  defined by Morabito.   Since the Hamiltonian symplecticmorphisms are degenerate in our situation, the idea is more close to  M. Khanevsky's generalization \cite{MK}.

Before giving the satement of our result, let us review   Morabito's construction \cite{FM}.  Let $\underline{L}=\cup_{i=1}^k L_i$ be a  link (a disjoint union of embedded circles) on  the disk $\mathbb{D}$.  Consider the group
\begin{equation*}
Ham_{\underline{L}}(\mathbb{D}, \omega): =\{\varphi \in Ham_c(\mathbb{D}, \omega) \vert \varphi(\underline{L}) = \underline{L}\},
\end{equation*}
where $Ham_c(\mathbb{D}, \omega)$ consists of the Hamiltonian  symplecticmorphisms with compact  support. 
Let $\mathbf{x} =(x_1, ...x_k) \in \underline{L}$ be a fixed base point, where $x_i \in L_i$.  Given $\varphi \in Ham_{\underline{L}}(\mathbb{D}, \omega) $, let $H$ be a Hamiltonian function generated $\varphi$.   
Let $\gamma_{\varphi} =\{\varphi_H^t(x_i)\}_{i=1}^k$. Note that $\varphi(x_i) \in L_{\sigma(i)}$ because $\varphi(\underline{L}) = \underline{L}$, where $\sigma$  is a permutation. Choose paths $\{p_i(t)\}_{i=1}^k$ such that $p_i \subset L_{\sigma(i)}$ and $p_i$ connecting $\varphi(x_i)$ and $x_{\sigma(i)}$. Concatenate $\gamma_{\varphi}$ and  $\{p_i(t)\}_{i=1}^k$;  then we get a braid $\bar{\gamma}_{\varphi}$.  The  braid   type of $\bar{\gamma}_{\varphi}$  is denoted by $b(\varphi, \underline{L})$. Note that  $b(\varphi, \underline{L})$ is independent of the choice of the path $\{p_i(t)\}_{i=1}^k$ and  the base point. Because $\pi_1(Ham_c(\mathbb{D}, \omega)) = 0$,  $b(\varphi, \underline{L})$ is also independent of the choice of the  Hamiltonian function $H$.

To apply the Floer homology, we view the disk $\mathbb{D}$ as  a domain  in   $\mathbb{S}^2-\{p\},$ where $p$ is the south pole. Define $Ham_{\underline{L}}(\mathbb{S}^2, \omega): =\{\varphi \in Ham(\mathbb{S}^2, \omega) \vert \varphi(\underline{L}) = \underline{L}\}.$  Then $Ham_{\underline{L}}(\mathbb{D}, \omega) \subset Ham_{\underline{L}}(\mathbb{S}^2, \omega)$ is the subgroup consisting  of Hamiltonian symplecticmorphisms supported  compactly in $\mathbb{D}$. The main result is as follows.

\begin{thm} \label{thm1}
Suppose that the link $\underline{L}$  is  $\eta$-admissible (see Definition \ref{def1}) with $k$ components. Then there exists a positive constant $\varepsilon_{\underline{L}}$ depending on $\underline{L}$  such that  the following property holds.
Given $\varphi \in Ham_{\underline{L}}(\mathbb{D}, \omega)$, 
for any   $\varphi'  \in Ham_{\underline{L}}(\mathbb{D}, \omega)$ such that $d_{Hofer}(\varphi, \varphi')<  \varepsilon_{\underline{L}}  /k$, then $ b(\varphi, \underline{L}) =  b(\varphi', \underline{L})$. 
\end{thm}

 We prove Theorem \ref{thm1}     by  using   the quantitative Heegaard Floer homology that is introduced by D. Cristofaro-Gardiner,  V. Humili$\grave{e}$re, C. Mak, S. Seyfaddini and I. Smith \cite{CHMSS}. One may  use Hutchings's holomorphic curve methods to prove the same result alternatively.

 Morabito defines a pseudonorm on $\mathcal{B}_k$ by
\begin{equation*}
|b|_{\underline{L}} : = \inf \{ |\varphi|_{Hofer} \vert  \varphi \in  Ham_{\underline{L}}(\mathbb{D}, \omega) \mbox{ such that } b(\varphi, \underline{L}) =b \}.
\end{equation*}
It is remarked by   Morabito that the morphism $Ham_{\underline{L}}(\mathbb{D}, \omega) \to \mathcal{B}_k, \varphi \to b(\varphi, \underline{L})$ is surjective (see Page 4 of \cite{FM}).  Therefore, the definition makes  sense for every $b \in \mathcal{B}_k. $
For $g,h \in  \mathcal{B}_k$, the pseudodistance induced by $|\cdot|_{\underline{L}}$ is $d_{\underline{L}}(g, h) := |g h^{-1}|_{\underline{L}}$.
      Morabito shows that $|\cdot|_{\underline{L}} $ is   indeed a norm when $k=2$ \cite{FM}. As an application of Theorem \ref{thm1}, we show that $| \cdot |_{\underline{L}}$ is actually a norm.
\begin{corollary}
Suppose that the link $\underline{L}$  is   $\eta$-admissible with $k$ components. Then the pseudonorm $| \cdot |_{\underline{L}}$ on $\mathcal{B}_k$ is nondegenerate.  Also, for any two different elements $g, h \in \mathcal{B}_k$,  we have  $d_{\underline{L}}(g, h)  \ge \varepsilon_{\underline{L}} /k$.  Moreover, if $\eta>0$, then $\mathcal{B}_k$ is unbounded with respect to $| \cdot |_{\underline{L}}$.

\end{corollary}
\begin{proof}
Suppose that $| b |_{\underline{L}} =0$. By definition,  we have  a sequence of  Hamiltonian symplecticmorphisms  $\varphi_n \in Ham_{\underline{L}}(\mathbb{D}, \omega)$ such that $b= b(\varphi_n, \underline{L}) $ and $ \lim_{n \to \infty}|\varphi_n|_{Hofer}= \lim_{n \to \infty}d_{Hofer}(\varphi_n, id)=0$.  By Theorem \ref{thm1}, we have $b= b(\varphi_n, \underline{L}) =  b(id, \underline{L})$ provided that $n$ is sufficiently large. The braid type $b(id, \underline{L})$ is represented by the constant braid, hence it is the unit in $\mathcal{B}_k$.
The same argument implies that  $d_{\underline{L}}(g, h)  \ge \varepsilon_{\underline{L}} /k$ for $g \ne h.$

The unboundedness of $\mathcal{B}_k$ is a direct consequence of  Theorem 1.4 of \cite{FM}.

\end{proof}

\ack{The author would like to  thank the anonymous referee for his/her detailed comments, suggestions and corrections have greatly improved this paper. The author is supported by the National Natural Science Foundation of China Youth Fund Project, Grant number: 12201106. }

\section{Quantitative Heegaard Floer homology}
In this section, we briefly review  the quantitative Heegaard Floer homology and its filtered version.   
Without loss of generality, we assume that $\int_{\mathbb{S}^2} \omega =1$ throughout.
\begin{definition} \label{def1}
Fix $\eta\ge 0$. A link   $\underline{L}=\cup_{i=1}^k L_i$ on the two-sphere  is called    $\eta$-admissible  if it satisfies the following conditions:
\begin{enumerate}
\item
$ \mathbb{S}^2 -\underline{L} = \cup_{i=1}^{k+1}  \mathring{B}_k$, where $\{\mathring{B}_i\}_{i=1}^k$ is a disjoint union of  open  disks. Let $B_i$ denote the closure  of $\mathring{B}_i$. Then    $\partial B—_i = L_i$ for $1\le  i\le k$, and ${B}_{k+1}$ is a planar domain with $k$ boundaries.
\item
$\lambda = \int_{B_i} \omega $ for $1\le  i\le k$  and $\lambda = 2\eta (k-1) + \int_{B_{k+1}} \omega.$
\end{enumerate}
\end{definition}
We assume that the  link is  $\eta$-admissible throughout.
\begin{remark}
The argument in this  paper still work if we get rid of the  second condition.  Using a little bit of work, the quantitative Heegaard Floer homology is still well defined. See Remark \ref{remark1} for more details.   

However,  say if $\int_{B_i} \omega \ne \int_{B_j} \omega $ for $i \ne j$, then  we must have $\varphi(B_i) =B_i$ because $\varphi \in Ham_{\underline{L}}(\mathbb{D}, \omega)$   preserves the area. 
Then  the Hamiltonian chord  $\{\gamma_i(t)\}_{i=1}^k$ satisfies $\gamma_i(0), \gamma_i(1) \in L_i$.  As a result,  the morphism $Ham_{\underline{L}}(\mathbb{D}, \omega) \to \mathcal{B}_k, \varphi \to b(\varphi, \underline{L})$ is not surjective. This point is pointed out by Morabito in Page 4 of \cite{FM}.
\end{remark}

 Let $X:=Sym^k \mathbb{S}^2$ and  $\Delta :=\{[x_1,..., x_k] \in X : \exists  \  x_i =x_j \mbox{ for } i \ne j\}$  be  the \textbf{diagonal}.  For  fixed points  $z_{i} \in \mathring{B}_{i}$, let 
 $$D_i: =\{[z_i, x_1,...,x_{k-1}] \in X: [x_1,...,x_{k-1}] \in Sym^{k-1} \mathbb{S}^2\}. $$
  Sometimes we write it as $\{z_i\} \times Sym^{k-1} \mathbb{S}^2 \subset X$ when we want to  emphasize  the fixed point $z_i$.  There is a symplectic form $\omega_X$  such that $\omega_X =Sym^k \omega$ outside a neighborhood  of the diagonal.

Let   $V$ be a  small neighborhood of $\Delta\cup D_{k+1}$.   Fix $\varphi=\varphi^1_H \in Ham(\mathbb{S}^2, \omega)$.   A Hamiltonian function  $\underline{H} \in C^{\infty}([0,1] \times X, \mathbb{R})$ is \textbf{compatible with $H$} if it satisfies the following conditions:
\begin{enumerate}
\item
$\underline{H}_t =Sym^k H_t $ outside   $V$.
\item
$\underline{H}_t$ is a ($t$-dependent) constant near $\Delta$.
\item
$\underline{H}_t$ is a ($t$-dependent) constant near  $D_{k+1}$.
\item
$\min_X Sym^kH_t\le  \underline{H}_t  \le \max_X Sym^kH_t$.
\end{enumerate}
The last two conditions  are  not necessary for defining the quantitative  Heegaard Floer homology. 
The propose of the third item is to find holomorphic curves contained in  $Sym^k(\mathbb{S}^2 -\{z_{k+1}\})$.      If the fourth item is true, then  $|\underline{H}|_{(1, \infty)} \le k|H|_{(1, \infty)}$. We use this property in the energy estimate. Note that the function $\underline{H}$ satisfying the four conditions always exists (see Remark 6.8 of \cite{CHMSS}).
\begin{remark}
Since we will consider $\varphi_H^t$ with compact  support in $\mathbb{D}$ and take $z_{k+1}$ outside $\mathbb{D}$, then we can choose $V$ sufficiently small such that $Sym^k \varphi_{{H}}^t(  \underline{L}) \cap V =\emptyset$   for $t \in [0, 1]$.   Therefore, the first condition implies that  $\varphi_{\underline{H}}^t(Sym^k \underline{L}) = Sym^k \varphi^t_H(\underline{L})$.  We assume that this is true throughout.
\end{remark}

Suppose that $\varphi$ is \textbf{nondegenerate} in the sense that $Sym^k\varphi(\underline{L})$ intersects   $Sym^k\underline{L}$ transversely.  Fix a base point $\mathbf{x} \in \underline{L}$. Let $\mathbf{y}: [0,1] \to X$ be a \textbf{Hamiltonian chord,} i.e., $\partial_t \mathbf{y}(t) =X_{\underline{H}} \circ \mathbf{y}(t)$ and $\mathbf{y}(0), \mathbf{y}(1) \in Sym^k \underline{L}$.  A capping  is a smooth  map $\hat{\mathbf{y}} : [0,1]_s \times [0,1]_t \to X$ such that $\hat{\mathbf{y}} (0, t) ={\mathbf{y}}(t) $, $\hat{\mathbf{y}}(1, t) ={\mathbf{x}} $ and $\hat{\mathbf{y}}(s, i) \in Sym^k \underline{L}$, where $i=0, 1$. Let $ \pi_2(\mathbf{x}, \mathbf{y})$ denote the set of cappings. We abuse the same notation  $\hat{\mathbf{y}}$ to denote its  equivalent  class in $\pi_2(\mathbf{x}, \mathbf{y}) / \ker \omega_X. $ 

 Define  a  complex ${CF}( \underline{L}, \underline{H}_t) $  to be the set of  formal  sums of equivalent  classes    of capping in $\pi_2(\mathbf{x}, \mathbf{y}) / \ker \omega_X $
\begin{equation}
 \sum_{(\mathbf{y}, \hat{\mathbf{y}})}  a_{(\mathbf{y}, \hat{\mathbf{y}})}(\mathbf{y}, \hat{\mathbf{y}})
 \end{equation} satisfying that  $a_{(\mathbf{y}, \hat{\mathbf{y}})} \in \mathbb{Z}_2 $ and for any $C\in \mathbb{R}$, there are only finitely $(\mathbf{y}, \hat{\mathbf{y}})$ such that  $\mathcal{A}_H(\mathbf{y}, \hat{\mathbf{y}}) >C$ and  $ a_{(\mathbf{y}, \hat{\mathbf{y}})}  \ne 0$, where  $\mathcal{A}_H$ is a functional on $\pi_2(\mathbf{x}, \mathbf{y})$ defined in $(\ref{eq4})$. 
\begin{definition}
An   $\omega_X$-tame  almost complex structures $J$  is called  {nearly-symmetric} if $J =Sym^k j$ near $ \Delta \cup \cup_{i=1}^{k+1}D_i$, where $j$ is a fixed complex structure on the sphere.
 \end{definition}
 Fix a generic path of nearly-symmetric almost complex structures $\{J_t\}_{t\in [0,1]}$.
A  $X_{\underline{H}_t }$-perturbed holomorphic strip is a solution to the following Floer's equations:
\begin{align}   \label{eq1}
\begin{split}
\left \{
\begin{array}{ll}
 \partial_s u + J_{t}(u)(\partial_t u -X_{\underline{H}_t } \circ u) =0\\
 u(s, 0), u(s, 1) \in Sym^k \underline{L}\\
  \lim_{s \to \pm \infty} u(s, t) =\mathbf{y}_{\pm}(t).
\end{array}
\right.
\end{split}
\end{align}
The moduli  space of index $i$ $X_{\underline{H}_t }$-perturbed holomorphic strips with homotopy class $\beta$ is denoted by $\mathcal{M}_i^{J_t}(\mathbf{y}_+ , \mathbf{y}_-, \beta)$.  Later, we will replace the vector field $X_{\underline{H}_t }$ in Equations (\ref{eq1}) by a Hamiltonian vector field $X$ which depends on $s$ or other parameters.  The solutions are called \textbf{$X$-perturbed holomorphic strips}.

Then the differential is defined by
\begin{equation} \label{eq2}
\partial ({\mathbf{y}}_+, \hat{\mathbf{y}}_+) := \sum_{\beta \in \pi_2(X,  \mathbf{y}_+, \mathbf{y}_-)} \#_2 \left(\mathcal{M}_1^{J_t}(\mathbf{y}_+ , \mathbf{y}_-, \beta) /\mathbb{R} \right) (\mathbf{y}_-,  \hat{\mathbf{y}}_+ \# \beta ).
\end{equation}
The proof of Lemma 6.6 in \cite{CHMSS} shows that $\partial^2=0$. The quantitative  Heegaard Floer homology is denoted by $HF(\underline{L}, H).$

\begin{remark}
In the definition of the capping classes, 
the symplectic form we used is $\omega_X$,  rather than $\omega_X + \eta PD_{\Delta}$.  This is  a subtle difference   comparing with \cite{CHMSS}. The reason is that  later we use the action functional  $\mathcal{A}_H$ (see \ref{eq4}) 
to define the filtration on the chain  complex. If we use $\omega_X + \eta PD_{\Delta}$ to define the capping classes, then the value of $\mathcal{A}_H$ at $(\mathbf{y}, \hat{\mathbf{y}})$  depends on the choice of the representatives    of $\hat{\mathbf{y}}$ because  $\ker(\omega_X + \eta PD_{\Delta}) \ne \ker \omega_X$, unless $\eta=0$.  
If we use  $\omega_X + \eta PD_{\Delta}$ to define the capping classes,  then we need to use the action functional  $\mathcal{A}^{\eta}_H$   (see (52) of \cite{CHMSS}). 

 For both of the definitions, we will end up  with  the same result.  The reason is that finally we only consider a filtered  subcomplex $({CF}^{(a,b)}(\underline{L}, \underline{H}), \partial_{00})$, where $\partial_{00}$ counts the holomorphic strips in $Sym^k(\mathbb{S}^2-\{z_{k+1}\}) -\Delta$, and the  holomorphic strips contributing    to $\partial_{00}$  are the same for using  $\omega_X + \eta PD_{\Delta}$ or  $\omega_X$.  Here ``the same'' also means that the energy of the holomorphic strips computed by       $\omega_X$ and $\omega_X + \eta PD_{\Delta}$ are equal. 
  The only difference is that the ``gap'' of the spectrum  is different.  More precisely, the constant  $\lambda_{\underline{L}}$ obtained  in Lemma 3.5  is different.   The core of our argument is that if $b-a$ is less than  suitable proportion  of $\lambda_{\underline{L}}$,  then  the differential $\partial_{00}$ only counts Morse flow lines,  and  the Floer homology  $\widehat{HF}^{(a,b)}(\underline{L}, \underline{H})$ (the homology of  $({CF}^{(a,b)}(\underline{L}, \underline{H}), \partial_{00})$) is nonvanishing and independent of $\underline{H}$. For two Hamiltonian functions with sufficiently close Hofer distance, the continuous morphism on   $\widehat{HF}^{(a,b)}$  is an isomorphism, and this gives us the desired holomorphic strip.  This argument doesn't rely on the precise value of $\lambda_{\underline{L}}$. 
\end{remark} 

\begin{remark} \label{remark1}
The Floer homology  $HF(\underline{L}, H)$ is still well defined  in a weaker setting that second condition in Definition \ref{def1} is dropped.  The reasons are as follows. 

  Consider  a nonconstant  holomorphic disk $u: (\mathbb{D}, \partial \mathbb{D}) \to (X, Sym^k \underline{L})$.     We claim that $\mu(u) \ge 2$, where $\mu(u)$ is the Maslov index. 
Let $\{u_i\}_{i=1}^{k+1}$ be the generators in  $\pi_2(X, Sym^k \underline{L})$. Note that these generators still satisfy 
$$\mu(u_i) =2 \mbox{ and } u_i \cdot D_j =\delta_{ij}. $$
This is because $\underline{L}$ is smoothly isotopic to an admissible link and the above properties are preserved under the isotopy. 

Write  $[u] =\sum_{i=1}^{k+1} a_i u_i$, where $a_i \in \mathbb{Z}$.  Since $u$ is holomorphic, the intersection positivity implies that $a_i = u \cdot D_i \ge 0$ (need to choose $J=Sym^kj$ near $D_i$). The case  that $a_i=0$ for all $ 0\le i\le  k+1$ cannot happen; otherwise, $u$ is a constant.  Therefore, $\mu(u) =2(\sum_{i=1}^{k+1} a_i) \ge 2$.  The argument here also implies that $\mu(u) =2 $ if and only if $[u]=u_i$ for some $i$. 
 
The property $\mu(u) \ge 2$ allows  us to rule out the bubbles when defining the differential $\partial$.   Also, the  broken holomorphic curves contributing  to $\partial^2$ are either a two level building with two index 1 strips, or  a constant strip attached with a disk bubble.  In the latter case, the disk bubble has Maslov index 2 by index reason.  By the  argument in Lemma 5.1 of \cite{CHMSS},  $\mathcal{M}^J(Sym^k \underline{L}, u_i)$ is a compact manifold of expected dimension, where $\mathcal{M}^J(Sym^k \underline{L}, u_i)$ is the moduli space of holomorphic disks  with one boundary marked point and in the  homology class $u_i$. Let $ev_{i, J}:  \mathcal{M}^J(Sym^k \underline{L}, u_i) \to Sym^k \underline{L} $ be the evaluation maps.  Then the standard  gluing argument implies 
$$\partial^2 (\mathbf{y}, \hat{\mathbf{y}}) = \sum_{i=1}^{k+1} \left(\# ev_{i, J_0}^{-1}(p) - \# ev_{i, J_1}^{-1}(p) \right) (\mathbf{y}, \hat{\mathbf{y}}) =0.$$
Therefore, the homology  $HF(\underline{L}, H)$ is  well defined. 

\end{remark}

\paragraph{Continuous morphisms}
A \textbf{homotopy} from  $(\underline{H}_+, J_{t}^+)$ to  $(\underline{H}_-, J_t^-)$ is a pair consisting of a function $\underline{H}: \mathbb{R}_s \times [0,1]_t \times X \to \mathbb{R} $  and a family of almost complex structures  $\{J_{s,t}\}_{s \in \mathbb{R}, t\in [0,1]}$ such that
\begin{itemize}
\item
 $(\underline{H}_s, J_{s,t}) =(\underline{H}_+, J^+_t)$ when $s \ge 1$ and $(\underline{H}_s, J_{s,t}) =(\underline{H}_-, J_t^-)$ when $s\le -1$,
\item
$\underline{H}_s$ is ($s, t$-dependent) constant near the  diagonal and $D_{k+1}$,
\item
For each  $(s,t) \in \mathbb{R} \times [0,1]$, $J_{s, t}$ is nearly-symmetric.
\end{itemize}
A  typical way of  constructing  $\underline{H}_s$ is to define $$\underline{H}_s =\chi(s) \underline{H}_+ +  (1-\chi(s)) \underline{H}_-, $$ where $\chi: \mathbb{R} \to \mathbb{R}$  is a cut-off function such that $\chi(s) = 1$  when $s \ge 1$ and $\chi(s) =-1$ when $s \le -1$.

Let $\mathcal{M}_i^{J_{s,t}}(\mathbf{y}_+ , \mathbf{y}_-, \beta)$ denote the moduli space of  the   index $i$ $X_{\underline{H}_s}$-perturbed holomorphic strips.
Define a morphism $\phi(\underline{H}_s, J_{s,t}): CF(\underline{L}, \underline{H}_+) \to CF(\underline{L}, \underline{H}_-)$ on chain level by
 \begin{equation*}
\phi(\underline{H}_s, J_{s,t})   ({\mathbf{y}}_+, \hat{\mathbf{y}}_+) := \sum_{\beta \in \pi_2(X, \mathbf{y}_+, \mathbf{y}_-)} \#_2 \mathcal{M}_0^{J_{s,t}}(\mathbf{y}_+ , \mathbf{y}_-, \beta) (\mathbf{y}_-,  \hat{\mathbf{y}}_+ \# \beta ).
\end{equation*}
   The above map induces a homomorphism $\Phi(\underline{H}_+, \underline{H}_-): HF(\underline{L}, H_+) \to HF(\underline{L}, H_-)$.
   Moreover, the homomorphism $\Phi(\underline{H}_+, \underline{H}_-)$ only depend on $(\underline{H}_+, \underline{H}_-)$.  $\Phi(\underline{H}_+, \underline{H}_-)$ is called a \textbf{continuous morphism.}

\paragraph{Filtration}
The action functional on $CF(\underline{L}, \underline{H})$ is defined by
\begin{equation} \label{eq4}
\mathcal{A}_{\underline{H}} ({\mathbf{y}}, \hat{\mathbf{y}}) := -\int \hat{\mathbf{y}}^* \omega_X + \int \underline{H}_t (\mathbf{y}(t))dt.
\end{equation} 
Since the differential decreases  the action,  $CF^{<a}(\underline{L}, \underline{H})$ is a subcomplex.  For $a<b$,  define   a  quotient  subcomplex $CF^{(a, b)}(\underline{L},  \underline{H}) :=CF^{<b}(\underline{L},  \underline{H})/CF^{<a}(\underline{L},  \underline{H})$. We assume that $a, b$ do not belong to the spectrum throughout.

There is another filtration on $CF(\underline{L}, \underline{H})$ induced by the intersection numbers with $\Delta$ and $D_{k+1}$.  Define
  \begin{equation*}
\partial_{ij} ({\mathbf{y}}_+, \hat{\mathbf{y}}_+) := \sum_{\beta \in \pi_2(X, \mathbf{y}_+, \mathbf{y}_-), \beta \cdot D_{k+1} =i,  \beta \cdot \Delta = j} \#_2 \left(\mathcal{M}_1^{J_t}(\mathbf{y}_+ , \mathbf{y}_-, \beta) /\mathbb{R} \right) (\mathbf{y}_-,  \hat{\mathbf{y}}_+ \# \beta ).
\end{equation*}
Note that $u$ is $Sym^k j$-holomorphic near $\Delta$  and $D_{k+1}$ because $X_{\underline{H}_t }=0$  there.  Also, $\Delta$ and $D_{k+1}$ are codimension two complex  varieties.  Therefore, $u\cdot \Delta \ge 0$ and $u\cdot D_{k+1} \ge 0$.
By definition, we have
\begin{equation*}
\partial= \partial_{00} + \partial_{10}+\partial_{01} ....
\end{equation*}
Then $0=\partial^2=\partial_{00}^2 + \partial_{10} \partial_{00}+ \partial_{00} \partial_{10} +...$ implies that $\partial_{00}^2=0$,  $ \partial_{10} \partial_{00}+ \partial_{00} \partial_{10} =0 $, etc.
Define $\widehat{HF}(\underline{L}, \underline{H})$ to be the homology of  $(CF(\underline{L}, \underline{H}), \partial_{00})$. Roughly speaking, $\widehat{HF} (\underline{L}, \underline{H})$ is the Lagrangian Floer homology of $Sym^k \underline{L}$ in $Sym^k(\mathbb{S}^2-\{z_{k+1}\})- \Delta$.

 It is easy to show that $\partial_{00}$ decreases  the action functional. The homology of   $(CF^{(a, b)}(\underline{L}, \underline{H}), \partial_{00})$ is denoted by $\widehat{HF}^{(a, b)}(\underline{L}, H)$.

Similarly, we have $ \phi(\underline{H}_s, J_{s,t})= \phi_{00}(\underline{H}_s, J_{s,t}) + \phi_{10}(\underline{H}_s, J_{s,t}) +\phi_{01}(\underline{H}_s, J_{s,t}) ...$. The chain map condition  $\partial \circ \phi =\phi \circ \partial $ implies that
\begin{equation*}
\sum_{i+i'=k} \sum_{j+j' =l} \left(\partial_{ij} \circ \phi_{i'j'}(\underline{H}_s, J_{s,t}) - \phi_{i'j'}(\underline{H}_s, J_{s,t}) \circ \partial_{ij} \right) =0
\end{equation*}
for each nonnegative integers $k, l$.  Here $i, j, i', j'$ are nonnegative. In particular, we have $$\partial_{00} \circ  \phi_{00}(\underline{H}_s, J_{s,t}) =\phi_{00} (\underline{H}_s, J_{s,t})\circ \partial_{00}. $$  The following lemma tells us that  $\phi_{00}(\underline{H}_s, J_{s,t})$ induces a homomorphism $$\Phi_0(\underline{H}_+, \underline{H}_-): \widehat{HF}(\underline{L}, H_+) \to \widehat{HF}(\underline{L}, H_-) $$  that only depends  on $(\underline{H}_+, \underline{H}_-)$. Moreover,   $ \widehat{HF}(\underline{L}, H)$ is independent of the choice of $H$ and almost complex structures.

\begin{lemma} \label{lem2}
The continuous morphisms satisfy the following properties:
\begin{enumerate}
\item (Chain homotopy)
Let $(\underline{H}^0_s, J^0_{s, t})$ and  $(\underline{H}^1_s, J^1_{s, t})$ be two generic homotopies from $(\underline{H}_+, J^+_t)$ to $(\underline{H}_-, J^-_t)$.   
Then there exists  a homomorphism
$K_{00}: CF(\underline{L}, \underline{H}_+) \to CF(\underline{L}, \underline{H}_-)$  such  that
$$\phi_{00}(\underline{H}_s^1, J_{s,t}^1)- \phi_{00}(\underline{H}_s^0, J_{s,t}^0) =\partial_{00} \circ K_{00} + K_{00} \circ \partial_{00}.$$

\item (Composition rule)
Let $(\underline{H}^0_s, J^0_{s, t})$  be a homotopy from $(\underline{H}_+, J^+_t)$ to $(\underline{H}_0, J^0_t)$ and  $(\underline{H}^1_s, J^1_{s, t})$  be a homotopy from  $(\underline{H}_0, J^0_t)$ to $(\underline{H}_-, J^-_t)$.  
 Concatenating $(\underline{H}^0_s, J^0_{s, t})$  and $(\underline{H}^1_s, J^1_{s, t})$  produces a homotopy  $(\underline{H}^1_s\circ \underline{H}^0_s, J_{s,t}^{1} \circ J^{0}_{s,t})$ from   $(\underline{H}_+, J^+_t)$ to $(\underline{H}_-, J^-_t)$.  Then there exists a homomorphism
$K_{00}: CF(\underline{L}, H_+) \to CF(\underline{L}, H_-)$ such  that
$$\phi_{00}(\underline{H}^1_s, J_{s,t}^1 )\circ \phi_0(\underline{H}^0_s, J_{s,t}^0) =\phi_{00}( \underline{H}^1_s\circ \underline{H}^0_s, J_{s,t}^{1} \circ J^{0}_{s,t}) + \partial_{00} \circ K_{00} + K_{00} \circ \partial_{00}.$$
\end{enumerate}

\end{lemma}
\begin{proof}
Let $\{(\underline{H}_s^{\tau}, J_{s,t}^{\tau})\}_{\tau \in [0,1]}$ be a  generic homotopy between  $(\underline{H}^0_s, J^0_{s, t})$  and  $(\underline{H}^1_s, J^1_{s, t})$ such that  $\underline{H}_s^{\tau}$ is constant and $J^{\tau}_{s,t} =Sym^k j$ near $\Delta$ and $D_{k+1}$.  In the usual Lagrangian Floer theory, we have
\begin{equation} \label{eq3}
\phi(\underline{H}_s^1, J_{s,t}^1)- \phi(\underline{H}_s^0, J_{s,t}^0) =\partial \circ K + K \circ \partial,
\end{equation}
where $K$ is defined by counting  the  $X_{\underline{H}^{\tau}_s}$-perturbed strips with index $-1$.   Since $X_{\underline{H}^{\tau}_s} =0$ and $J_{s,t}^{\tau} =Sym^k j$ near $\Delta$ and $D_{k+1}$, we still have $u \cdot \Delta \ge 0$ and $ u \cdot D_{k+1} \ge 0$. Therefore, we can decompose $K =K_{00} + K_{10} + K_{01} ...$.  Then  Equation (\ref{eq3}) implies that
\begin{equation*}
\phi_{kl}(\underline{H}_s^1, J_{s,t}^1)- \phi_{kl}(\underline{H}_s^0, J_{s,t}^0)  = \sum_{i+i'=k} \sum_{j+j'=l} \left(\partial_{ij} \circ K_{i'j'} - K_{i'j'} \circ \partial_{ij} \right).
\end{equation*}
The first statement is just the special case that $k=l=0$.

The proof of the second statement is similar.
\end{proof}

\begin{remark}
By P. Biran and O. Cornea's criterions (Proposition 6.1.4 of \cite{BC})   and the computations in \cite{CHMSS}, one should  expect that $\widehat{HF}(\underline{L}, H)=0. $ But we don't need this result because we only use the filtered version   $\widehat{HF}^{(a, b)}(\underline{L}, H)$.
\end{remark}

 \begin{remark}
In \cite{GHC}, the author gives an alternative formulation   of $HF(\underline{L}, H)$ by counting the  holomorphic curves in  $\mathbb{R} \times [0,1] \times \Sigma$ which are  called HF curves.
Under the tautological correspondence, the intersection  number $u \cdot \Delta $ corresponds to the $J_0$ index of HF curves in $\mathbb{R} \times [0,1] \times \Sigma$ (see Proposition 4.2 of \cite{GHC}). The $J_0$ index is the counterpart of Hutchings's one in ECH setting  that measures the  Euler characteristic of holomorphic curves \cite{H1, H2}.  If an HF curve has zero $J_0$ index, then it is a disjoint union of holomorphic strips.
\end{remark}

\section{Proof of Theorem \ref{thm1}}
To apply the quantitative Heegaard Floer homology, we need to perturb $\varphi^1_{\underline{H}}$ 
 so that it is nondegenerate. Fix a perfect Morse function  $f_{\underline{L}}: Sym^k \underline{L} \to \mathbb{R}$. Extend $f_{\underline{L}}$  to a   smooth function $f$ on $X$ such that
\begin{itemize}
\item
$f$ is Morse outside a small neighborhood of $\Delta$ and $D_{k+1}$;
\item
A critical point of $f_{\underline{L}}$ is also a critical point of $f$;
\item
$f$ is a  constant near $\Delta$ and $D_{k+1}$.
\end{itemize}
Denote
$\underline{H} \# \epsilon f = \underline{H} +  \epsilon f\circ (\varphi^t_{\underline{H} })^{-1}$ by  $\underline{H}^{\epsilon}$. Note that $\underline{H}^{\epsilon}$  generates  $\varphi^t_{\underline{H}} \circ \varphi_{\epsilon f}^t$ and  $\underline{H}^{\epsilon}$  is  constant near $\Delta$ and $D_{k+1}$.  The next two lemmas tell us that the chain complex $(CF(\underline{L}, \underline{H}^{\epsilon}), \partial)$ is identical to   $(CF(\underline{L},  {\epsilon} f), \partial)$ provided that $\epsilon$ is sufficiently small.

\begin{lemma}
There exists a constant $\epsilon_{\underline{L}, f}^1>0$ depending only on the link and $f$ such that for   $0<\epsilon \le \epsilon_{\underline{L}, f}^1$, $\varphi_{\underline{H}}^1 \circ \varphi_{\epsilon f}^1$ is nondegenerate.  Also, the Hamiltonian chords of  $\underline{H}^{\epsilon}$ are  of the form $\mathbf{y}(t) =\varphi^t_{\underline{H}}(\mathbf{x})$, where $\mathbf{x} \in Crit(f_{\underline{L}})$.
\end{lemma}
\begin{proof}
For any $\mathbf{x}\in Sym^k \underline{L}$, $\varphi_{\underline{H}}^1 \circ \varphi_{\epsilon f}^1(\mathbf{x}) \in Sym^k \underline{L}$ if and only if $\mathbf{x} \in Sym^k\underline{L} \cap (\varphi_{\epsilon f}^1)^{-1}(Sym^k\underline{L}) $ because $\varphi_H^1(\underline{L}) =\underline{L}$. Also,  note that $\varphi_{\underline{H}}^1 \circ \varphi_{\epsilon f}^1$ is nondegenerate if and only if $(\varphi_{\epsilon f}^1)^{-1}(Sym^k\underline{L}) $ intersects  $ Sym^k\underline{L} $ transversality.
By the proof of Lemma 6.1 of \cite{GHC}, there exists $\epsilon_{\underline{L}, f}^1>0$ such that if $0< \epsilon \le \epsilon_{\underline{L}, f}^1$, then  $(\varphi_{\epsilon f}^1)^{-1}(Sym^k\underline{L}) $ intersects  $ Sym^k\underline{L} $  transversally and  $\mathbf{x} \in Crit(f_{\underline{L}})$.

Since $\mathbf{x}$ is also a critical point of $f$,  we have $ \varphi_{\epsilon f}^t(\mathbf{x}) =\mathbf{x}$. Therefore, the Hamiltonian chord satisfies $ \varphi^t_{\underline{H}} \circ  \varphi_{\epsilon f}^t(\mathbf{x})  =\varphi^t_{\underline{H}}(\mathbf{x}).$

\end{proof}

Let  $\mathbf{x}  \in Sym^k \underline{L}$ be  a critical point of $f_{\underline{L}}$.    Using the Hamiltonian chord  of $\varphi_{\underline{H}}^t \circ \varphi_{\epsilon f}^t(\mathbf{x})$, we construct a braid $\bar{\gamma}_{\varphi \circ \varphi_{\epsilon f}^1}$ as before. Because $ \varphi_{\underline{H}}^t \circ \varphi_{\epsilon f}^t(\mathbf{x}) =  \varphi_{\underline{H}}^t (\mathbf{x})$,  the braid type   of $\bar{\gamma}_{\varphi \circ \varphi_{\epsilon f}^1}$  still  is $b(\varphi, \underline{L})$.

\begin{lemma} \label{lem5}
Assume that  $0<\epsilon \le \epsilon_{\underline{L}, f}^1$. Let  $\mathbf{y}_{\pm}(t) =\varphi^t_{\underline{H}}(\mathbf{x}_{\pm})$, where $\mathbf{x}_{\pm} \in Crit(f_{\underline{L}})$. Fix a  nearly symmetric almost complex structure $J_0$. Let $J_t := (\varphi_{\underline{H}}^t)_* \circ J_0 \circ  (\varphi_{\underline{H}}^t)_*^{-1}$. Then there is a bijection
$$\Psi: \mathcal{M}^{J_t}(\mathbf{y}_+, \mathbf{y}_-) \to  \mathcal{M}^{J_0}(\mathbf{x}_+, \mathbf{x}_-),$$
where $\mathcal{M}^{J_t}(\mathbf{y}_+, \mathbf{y}_-)$ is the moduli space of  $X_{\underline{H}^{\epsilon}}$-strips and  $ \mathcal{M}^{J_0}(\mathbf{x}_+, \mathbf{x}_-)$ is the moduli space of $X_{\epsilon f}$-strips. Moreover, $\Psi$ preserves the energy in the sense that
\begin{equation*}
\int u^*\omega_X  + \int \underline{H}^{\epsilon}(t, \mathbf{y}_+(t))dt -  \int \underline{H}^{\epsilon}(t, \mathbf{y}_-(t))dt = \int \Psi(u)^*\omega_X + \epsilon f(\mathbf{x}_+) - \epsilon f(\mathbf{x}_-).
\end{equation*}
\end{lemma}
\begin{proof}
For $u \in  \mathcal{M}^{J_t}(\mathbf{y}_+, \mathbf{y}_-)$, define the map $\Psi$  by sending $u$ to $v(s,t) : =(\varphi_{\underline{H}}^t)^{-1}(u(s,t ))$.  Note that $\lim_{s \to \pm \infty }v(s, t) =(\varphi_{\underline{H}}^t)^{-1}( \mathbf{y}_{\pm}(t)) = \mathbf{x}_{\pm}, $ $v(s, 0) = u(s, 0) \subset Sym^k{\underline{L}}$, and $v(s, 1) =(\varphi_{\underline{H}}^1)^{-1} (u(s, 1)) \subset (\varphi_{\underline{H}}^1)^{-1}(Sym^k{\underline{L}}) = Sym^k{\underline{L}}$.

By a direct computation, we have
 \begin{equation*}
\begin{split}
& \partial_s v = (\varphi_{\underline{H}}^t)^{-1}_*(\partial_su),  \partial_t v = (\varphi_{\underline{H}}^t)^{-1}_*(\partial_t u - X_{\underline{H}} \circ u) \\
&X_{\underline{H}^{\epsilon }} = X_{\underline{H} }  +  (\varphi_{\underline{H}}^t)_*(X_{\epsilon f} \circ  (\varphi_{\underline{H}}^t)^{-1}).
\end{split}
 \end{equation*}
 Therefore, $\partial_s v + J_0(v)(\partial_t v -X_{\epsilon f} \circ v) = (\varphi_{\underline{H}}^t)^{-1}_*( \partial_s u + J_t(u) (\partial_t u -X_{ \underline{H}^{\epsilon}} \circ u)) = 0$.

In sum, $\Psi$ is well-defined. Obviously, it is a bjiection.  Let $v=\Psi(u)$. We have
 \begin{equation*}
\begin{split}
\int u^* \omega_X& =\int \omega_X \left( (\varphi_{\underline{H}}^t)_*(\partial_s v),  (\varphi_{\underline{H}}^t)_*(\partial_t v) + X_{\underline{H}} \circ u \right) ds \wedge dt \\
&= \int v^* \omega_X + \int \omega_X( (\varphi_{\underline{H}}^t)_*(\partial_s v), X_{\underline{H}^{\epsilon }} \circ u- (\varphi_{\underline{H}}^t)_*(X_{\epsilon f} \circ v ) ) ds \wedge dt \\
&= \int v^* \omega_X +   \int \omega_X(  \partial_su, X_{\underline{H}^{\epsilon }} \circ u) ds \wedge dt   - \int \omega_X(\partial_s v, X_{\epsilon f} \circ v) ds \wedge dt \\
&= \int v^* \omega_X -  \int  d{\underline{H}^{\epsilon }} (\partial_s u) ds \wedge dt  - \int   d({\epsilon f} ) (\partial_s v) ds \wedge dt \\
&= \int v^* \omega_X - \int \underline{H}^{\epsilon}(t, \mathbf{y}_+(t))dt  +  \int \underline{H}^{\epsilon}(t, \mathbf{y}_-(t))dt + \epsilon f(\mathbf{x}_+) - \epsilon f(\mathbf{x}_-).
\end{split}
 \end{equation*}
So $\Psi$ preserves the energy.
\end{proof}

The next lemma describes  how the    filtered complexes change under the continuous morphisms.  The proof is just the standard energy estimates. 
\begin{lemma} \label{lem1}
Fix $a< b$.  Let $H_{\pm}$ be Hamiltonian functions with compact  support. Suppose that  $|H_+-H_-|_{(1, \infty)} < \varepsilon/2k$ and $\epsilon|f|_{C^0} < \varepsilon/4$, then the continuous morphism  $\phi_{00}(\underline{H}_s, J_{s,t})$ descends to the following maps:
\begin{equation*}
\begin{split}
&\phi_{00}(\underline{H}_s, J_{s,t}): CF^{<a}(\underline{L}, \underline{H}_+^{\epsilon}) \to CF^{<a+ \varepsilon}(\underline{L}, \underline{H}^{\epsilon}_-) \\
&\phi_{00}(\underline{H}_s, J_{s,t}): CF^{(a, b)}(\underline{L}, \underline{H}^{\epsilon}_+) \to CF^{(a+ \varepsilon, b + \varepsilon)}(\underline{L}, \underline{H}^{\epsilon}_-).\\
\end{split}
\end{equation*}
\end{lemma}
\begin{proof}
 We only prove the first statement because the second statement is a consequence of the first one.

Let $u \in   \mathcal{M}^{J_{s, t}}(\mathbf{y}_+ , \mathbf{y}_-, \beta)  $ be a $X_{\underline{H}_s}$-perturbed holomorphic strip that contributes to the continuous morphism.  Then we have
 \begin{equation*}
\begin{split}
&\mathcal{A}_{\underline{H}^{\epsilon}_+} ({\mathbf{y}}_+, \hat{\mathbf{y}}_+)  -\mathcal{A}_{\underline{H}^{\epsilon}_-} ({\mathbf{y}}_-, \hat{\mathbf{y}}_-) \\
 =& \int u^*\omega_X  +  \int_0^1 \underline{H}^{\epsilon}_+(t, \mathbf{y}_+(t))dt -  \int_0^1 \underline{H}^{\epsilon}_-(t, \mathbf{y}_-(t))dt\\
=& \int \omega_X(\partial_s u , \partial_t u ) ds \wedge dt +  \int_0^1 \underline{H}^{\epsilon}_+(t, \mathbf{y}_+(t))dt -  \int_0^1 \underline{H}^{\epsilon}_-(t, \mathbf{y}_-(t))dt\\
=& \int (\omega_X(\partial_s u , J_{s,t}(\partial_s u) )  + \omega_X(\partial_su, X_{\underline{H}_s}))ds \wedge dt +  \int_0^1 \underline{H}^{\epsilon}_+(t, \mathbf{y}_+(t))dt -  \int_0^1 \underline{H}^{\epsilon}_-(t, \mathbf{y}_-(t))dt\\
=& \int (|\partial_s u|^2 - d{\underline{H}}_s(\partial_s u))ds \wedge dt +  \int_0^1 \underline{H}^{\epsilon}_+(t, \mathbf{y}_+(t))dt -  \int_0^1 \underline{H}^{\epsilon}_-(t, \mathbf{y}_-(t))dt\\
=&\int |\partial_s u|^2  ds \wedge dt -\int d(\underline{H}_s \circ u dt) + \int  (\partial_s \underline{H}_s) \circ u+   \int_0^1 \underline{H}^{\epsilon}_+(t, \mathbf{y}_+(t))dt -  \int_0^1 \underline{H}^{\epsilon}_-(t, \mathbf{y}_-(t))dt\\
=&\int |\partial_s u|^2  ds \wedge dt + \int ( \partial_s\underline{H}_s ) \circ u ds \wedge dt \\
\ge&   \int (\partial_s\underline{H}) \circ uds \wedge dt  \ge  \int_0^1 \min_{X}|\underline{H}^{\epsilon}_+ -\underline{H}^{\epsilon}_-|dt  \\
\ge& -k |H_+ -H_-|_{(1, \infty)} -2\epsilon|f|_{C^0} >-\varepsilon.
\end{split}
\end{equation*}
The above estimates imply that $\phi_{00}(\underline{H}_s, J_{s,t})$ maps  $CF^{<a}(\underline{L}, \underline{H}_+^{\epsilon})$ to $CF^{<a+ \varepsilon}(\underline{L}, \underline{H}^{\epsilon}_-)$.
\end{proof}



The following definition is an analogy  of  Definition 6  in \cite{MK}.
\begin{definition}
 Suppose that $\varepsilon>2\epsilon |f|_{C^0}$. We say that the spectrum of \textbf{$\mathcal{A}_{\underline{H}^{\epsilon}}$ is $\varepsilon$-admissible}  for $(\epsilon f, J_t)$ if we have
\begin{itemize}
\item
either $ |\mathcal{A}_{\underline{H}^{\epsilon}} ({\mathbf{y}}, \hat{\mathbf{y}}) - \mathcal{A}_{\underline{H}^{\epsilon}} ({\mathbf{y}}', \hat{\mathbf{y}}') | \ge \varepsilon$  or  $ |\mathcal{A}_{\underline{H}^{\epsilon}} ({\mathbf{y}}, \hat{\mathbf{y}}) - \mathcal{A}_{\underline{H}^{\epsilon}} ({\mathbf{y}}', \hat{\mathbf{y}}') | \le   2\epsilon |f|_{C^0}$;
\item
In the case that $ |\mathcal{A}_{\underline{H}^{\epsilon}} ({\mathbf{y}}, \hat{\mathbf{y}}) - \mathcal{A}_{\underline{H}^{\epsilon}} ({\mathbf{y}}', \hat{\mathbf{y}}') | \le   2\epsilon |f|_{C^0}$,   we have $$\#_2  (\mathcal{M}_1^{J_t}(\mathbf{y}_+, \mathbf{y}_-, \beta) / \mathbb{R}) =0,$$ where $\beta$ is the homotopy class such that $  \hat{\mathbf{y}} \# \beta  \# (-\hat{\mathbf{y}}') = 0$   in $\pi_2(X, Sym^k \underline{L})/\ker \omega_X$.

\end{itemize}
\end{definition}


\begin{lemma} \label{lem4}
Let $J_t =  (\varphi_{\underline{H}}^t)_* \circ J_0 \circ  (\varphi_{\underline{H}}^t)_*^{-1}$, where $J_0$ is a generic nearly symmetric almost complex structure. There exists positive constants   $\epsilon^2_{\underline{L}, f, J_0}$ (depending on the link,    $f$ and $J_0$) and $\lambda_{\underline{L}}$ (depending on $\underline{L}$) such that  if $0<\epsilon \le \epsilon^2_{\underline{L}, f, J_0}$, the spectrum of $\mathcal{A}_{\underline{H}^\epsilon }$ is $\lambda_{\underline{L}}$-admissible for $(\epsilon f, J_t)$. Moreover, given $ ({\mathbf{y}}= \varphi_{\underline{H}}^t(\mathbf{x}), \hat{\mathbf{y}}) $,  then we have
 \begin{equation*}
\# \{ ({\mathbf{y}}', \hat{\mathbf{y}}') :|\mathcal{A}_{\underline{H}^{\epsilon}} ({\mathbf{y}}, \hat{\mathbf{y}}) - \mathcal{A}_{\underline{H}^{\epsilon}} ({\mathbf{y}}', \hat{\mathbf{y}}') |     \le 2\epsilon |f|_{C^0} \} = 2^{k}.
 \end{equation*}
\end{lemma}
\begin{proof}
 Let $\mathbf{y}(t) =\varphi^t_{\underline{H}^{\epsilon}} (\mathbf{x})$ and  $\mathbf{y}'(t) =\varphi^t_{\underline{H}^{\epsilon}} (\mathbf{x}')$  be Hamiltonian chords, where $\mathbf{x}, \mathbf{x}' \in Crit(f_{\underline{L}})$.  Let $\eta: [0,1] \to Sym^k \underline{L}$ be a path from $\mathbf{x}'$ to $\mathbf{x}$.   Define  $u(s, t) : =\varphi_{\underline{H}}^t(\eta(s))$. Note that $u(1, t) = \mathbf{y}(t)$, $u(0, t) = \mathbf{y}'(t)$, $u(s, 0) =\eta(s) \subset Sym^k \underline{L}$ and  $u(s, 1) =\varphi_{\underline{H} }^1(\eta(s)) \subset Sym^k \underline{L}$ because $\varphi_H^1(\underline{L}) = \underline{L}$.  Then
 \begin{equation*}
\int u^* \omega_X = \int  \omega_X(\partial_s u, X_{\underline{H}} \circ u) ds\wedge dt = -\int d \underline{H} (\partial_s u) ds \wedge dt = \int_0^1 \underline{H}(\mathbf{y}'(t))dt -\int_0^1 \underline{H}(\mathbf{y}(t))dt.
 \end{equation*}
Therefore, we have
 \begin{equation*}
\begin{split}
\mathcal{A}_{\underline{H}^{\epsilon}}(\mathbf{y}, \hat{\mathbf{y}}) - \mathcal{A}_{\underline{H}^{\epsilon}}(\mathbf{y}', \hat{\mathbf{y}}') = -\int_A \omega_X + \epsilon f(\mathbf{x}) - \epsilon f(\mathbf{x}').
\end{split}
 \end{equation*}
Here the class $A :=  \hat{\mathbf{y}} \# [u] \# (-\hat{\mathbf{y}}') $ belongs to $\pi_2(X, Sym^k \underline{L})/\ker \omega_X$.  

By  Corollary  4.10 of \cite{CHMSS}, $\pi_2(X, Sym^k \underline{L})$ is freely generated by $\{u_i\}_{i=1}^{k+1}$ plus some torsion terms, where $u_i$ are  the  tautological correspondent  of $B_i$.  By the construction, we have  $\int u_i^* \omega_X = \int_{B_i} \omega$.  Define
\begin{equation*}
   \lambda_{\underline{L}}:=  \frac{1}{2} \min\{ \varepsilon>0 \vert  \exists   (a_1...a_k) \in \mathbb{Z}^{k+1} \mbox{ s.t } |\sum_{i=1}^{k+1} a_i \int_{B_i} \omega | = \varepsilon\}>0.
\end{equation*}
Then $|\int_A \omega_X| \ge 2 \lambda_{\underline{L}}$ if $\int_A \omega_X \ne 0.$ Let $\epsilon'_{\underline{L}, f} :=\min\{ \epsilon_{\underline{L}, f}^1,   \lambda_{\underline{L}}/10|f|_{C^0} \}$.  Assume that $0< \epsilon\le \epsilon'_{\underline{L}, f}.$
 If $\int_A \omega_X \ne 0$, then
 \begin{equation*}
|\mathcal{A}_{\underline{H}^{\epsilon}}(\mathbf{y}, \hat{\mathbf{y}})- \mathcal{A}_{\underline{H}^{\epsilon}}(\mathbf{y}', \hat{\mathbf{y}}') | \ge  2\lambda_{\underline{L}} -2\epsilon|f|_{C^0}  \ge \lambda_{\underline{L}}.  
 \end{equation*}
If $\int_A \omega_X =0$, then  $$|\mathcal{A}_{\underline{H}^{\epsilon}}(\mathbf{y}, \hat{\mathbf{y}}) - \mathcal{A}_{\underline{H}^{\epsilon}}(\mathbf{y}', \hat{\mathbf{y}}')| =  |\epsilon f(\mathbf{x}) - \epsilon f(\mathbf{x}')| \le 2\epsilon |f|_{C^0}.$$
 By Y.-G. Oh's results  (Proposition 4.1, 4.6) \cite{Oh},  there exists a constant $\epsilon''_{\underline{L}, f, J_0}$ such that for any $0<\epsilon \le \epsilon_{\underline{L}, f, J_0}''$, the index 1 $X_{\epsilon f}$-strips with energy smaller than $\lambda_{\underline{L}}$ are 1-1 corresponding to the index 1 Morse flow lines of $f_{\underline{L}}$. By Lemma \ref{lem5}, we have   $$\#_2  (\mathcal{M}_1^{J_t}(\mathbf{y}_+, \mathbf{y}_-, \beta) / \mathbb{R}) =0,$$ where  $\beta$ is the homotopy class such that $  \hat{\mathbf{y}} \# \beta  \# (-\hat{\mathbf{y}}') = 0$.  Here we use the assumption that $f_{\underline{L}}$ is perfect. Set $\epsilon^2_{\underline{L}, f, J_0} := \min\{\epsilon'_{\underline{L}, f, J_0}, \epsilon''_{\underline{L}, f}\}$. This is the constant in the statement.


Fix $(\mathbf{y}, \hat{\mathbf{y}})$. For any $ {\mathbf{y}}'$, there exists a unique $ \hat{\mathbf{y}}' \in \pi_2( {\mathbf{x}},  {\mathbf{y}}') / \ker \omega_X$ such that $A=0$. 
Therefore, we can deduce the second statement.
\end{proof}

\begin{proof} [Proof of Theorem \ref{thm1}]
Let $  \varepsilon =  \lambda_{\underline{L}}/100$ and $0<\epsilon \le \epsilon^2_{\underline{L}, f, J_0} $ such that $\epsilon|f|_{C^0} < \varepsilon/4$.
Fix  $\varphi_{+} \in Ham_{\underline{L}}(\mathbb{D}, \omega)$. Let $H_+$ be a Hamiltonian function   generating    $\varphi_+$ with compact support. Fix $ ({\mathbf{y}}_+, \hat{\mathbf{y}}_+)$, where $\mathbf{y}_+(t)=  \varphi^t_{\underline{H}_+}(\mathbf{x})$ for $\mathbf{x} \in Crit(f_{\underline{L}})$.  Let
\begin{equation*}
 a=\mathcal{A}_{\underline{H}^{\epsilon}_+} ({\mathbf{y}}_+, \hat{\mathbf{y}}_+) -5\varepsilon  \mbox{ and } b=\mathcal{A}_{\underline{H}^{\epsilon}_+} ({\mathbf{y}}_+, \hat{\mathbf{y}}_+) + 5\varepsilon.
\end{equation*}
 By Lemma \ref{lem4}, we have
\begin{equation*}
{CF}^{(a, b)}(\underline{L}, \underline{H}^{\epsilon}_+) = \mathbb{Z}_2 \{ ({\mathbf{y}}', \hat{\mathbf{y}}')  \vert  |\mathcal{A}_{\underline{H}^{\epsilon}_+} ({\mathbf{y}}_+, \hat{\mathbf{y}}_+) -\mathcal{A}_{\underline{H}^{\epsilon}_+} ({\mathbf{y}}', \hat{\mathbf{y}}') | \le 2\epsilon |f|_{C^0}    \}.
\end{equation*}
 The $\lambda_{\underline{L}}$-admissible conditions  also imply that  $\partial_{00}=0$ on  $CF^{(a, b)}(\underline{L}, \underline{H}^{\epsilon}_+) $. 
By Lemma \ref{lem4}, we have
\begin{equation*}
\widehat{HF}^{(a, b)}(\underline{L}, H^{\epsilon}_+) = \mathbb{Z}_2 \{ ({\mathbf{y}}', \hat{\mathbf{y}}')  \vert   |\mathcal{A}_{\underline{H}^{\epsilon}_+} ({\mathbf{y}}_+, \hat{\mathbf{y}}_+) -\mathcal{A}_{\underline{H}^{\epsilon}_+} ({\mathbf{y}}', \hat{\mathbf{y}}') | \le 2\epsilon |f|_{C^0}     \}=\mathbb{Z}_2^{2^k},
\end{equation*}
 By the same argument,  we have $\widehat{HF}^{(a+2\varepsilon, b+2\varepsilon)}(\underline{L}, H^{\epsilon}_+) = \mathbb{Z}_2^{2^k}$.

Set  $\varepsilon_{\underline{L}}: = \varepsilon/3  =\lambda_{\underline{L}} / 300$. Let $\varphi_- \in Ham_{\underline{L}}(\mathbb{D}, \omega) $   such that $d_{Hofer}(\varphi_+, \varphi_-)<  \varepsilon_{\underline{L}}/ k$.  For any $0<\delta<  \varepsilon/ 6k$,  we can find   a Hamiltonian function  $H_-$   generating  $\varphi_-$ with  compact support  such that $$|H_+ - H_-|_{(1, \infty)} \le d_{Hofer}(\varphi_+, \varphi_-) + \delta < \varepsilon/2k.$$
By Lemma \ref{lem1}, we have
\begin{equation*}
\mathbb{Z}_2^{2^k}= \widehat{HF}^{(a, b)}(\underline{L}, H^{\epsilon}_+)  \xrightarrow{\Phi_0(\underline{H}^{\epsilon}_+, \underline{H}^{\epsilon}_-)}  \widehat{HF}^{(a+\varepsilon, b+\varepsilon)}(\underline{L}, H^{\epsilon}_-)    \xrightarrow{\Phi_0(\underline{H}^{\epsilon}_-, \underline{H}^{\epsilon}_+)}  \widehat{HF}^{(a+2\varepsilon, b+2\varepsilon)}(\underline{L}, {H}^{\epsilon}_+) =\mathbb{Z}_2^{2^k}.
\end{equation*}
By the filtered version of Lemma \ref{lem2}, we have $\Phi_0(\underline{H}^{\epsilon}_+, \underline{H}^{\epsilon}_-)\circ \Phi_0(\underline{H}^{\epsilon}_-, \underline{H}^{\epsilon}_+) = \Phi_0(\underline{H}^{\epsilon}_+, \underline{H}^{\epsilon}_+)$.  We can use $({\underline{H}^{\epsilon}_+}, J_t)$ to define $\Phi_0(\underline{H}^{\epsilon}_+, \underline{H}^{\epsilon}_+)$.  By index reason, the $X_{\underline{H}^{\epsilon}_+}$-perturbed holomorphic strips   contributing to  $\phi_{00}({\underline{H}^{\epsilon}_+}, J_t)$   must be the trivial strips. Therefore, $ \Phi_0({\underline{H}^{\epsilon}_+}, \underline{H}^{\epsilon}_+)$ is the identity. 

As a result, $\Phi_0(\underline{H}^{\epsilon}_+, \underline{H}^{\epsilon}_-) $ is an  injection. 
Then  there is  a  $X_{\underline{H}_s}$-strip $u$ such that $\lim_{s \to \pm \infty} u(s, t) = \mathbf{y}_{\pm}(t)$ and $u \cdot \Delta =u\cdot D_{k+1}=0$. 
Recall that  we choose $J_{s,t} =Sym^kj $  and $\underline{H}_s$ to be constant  near $\Delta$ and $D_{k+1}$,  then  $u \cdot \Delta = u \cdot D_{k+1} =0$  implies that $u$  is disjoint from $\Delta$ and $\{z_{k+1}'\} \times Sym^{k-1} \mathbb{S}^2$, where $z_{k+1}'$ is sufficiently close to $z_{k+1}. $ Hence, $u$ lies inside $Sym^k (\mathbb{S}^2 -U_{z_{k+1}})$, where  $U_{z_{k+1}}$ is a sufficiently  small neighborhood of $z_{k+1}$.  

Let $\mathbf{p}_{\pm}$ be  paths in $Sym^k \underline{L}$ connecting $ \mathbf{y}_{\pm}(0)$ and $ \mathbf{y}_{\pm}(1)$.  Then we have braids $\bar{\gamma}_{\pm} = \mathbf{y}_{\pm} \cup (-\mathbf{p}_{\pm}). $ To produce a homotopy, we need to cap off   $ \mathbf{p}_+  \cup u \cup (-\mathbf{p}_-)$, where  $ \mathbf{p}_+  \cup u \cup (-\mathbf{p}_-)$    is the disjoin union $u \sqcup \mathbf{p}_+ \sqcup \mathbf{p}_-$ modulo out  the relations $(\mathbf{p}_{\pm}(0) \sim  u(\pm\infty, 0))$ and $(\mathbf{p}_{\pm}(1) \sim  u(\pm\infty, 1))$.  Strictly speaking, here we work with the compactification $u: \overline{\mathbb{R}} \times [0,1] \to X$, where $\overline{\mathbb{R}}  = \mathbb{R} \cup\{\pm \infty\}$. 

The precise meaning of capping  off $ \mathbf{p}_+  \cup u \cup (-\mathbf{p}_-)$ is as follows. Let $\partial u := u(\overline{\mathbb{R}} \times \{0,1\})$. Note that   $\partial u $ does not include the Hamiltonian chords.  Let $\mathbf{p}_{+} \cup \partial u \cup (-\mathbf{p}_{-})$ denote the concatenation of  $\mathbf{p}_{+}$, $ \partial u$ and $-\mathbf{p}_{-}$, i.e., the disjoin union $   \mathbf{p}_+ \sqcup   \partial u  \sqcup \mathbf{p}_-$ modulo out  the relations $(\mathbf{p}_{\pm}(0) \sim  u(\pm\infty, 0))$ and $(\mathbf{p}_{\pm}(1) \sim  u(\pm\infty, 1))$.  This forms a loop in $Sym^k \underline{L}$.  We want to show that $\mathbf{p}_{+} \cup \partial u \cup (-\mathbf{p}_{-})$ bounds a disk. Then gluing  this disk with $ \mathbf{p}_+ \cup u \cup (-\mathbf{p}_-)$ will produce a homotopy between the braids.

To this end, we first recall that $\pi_2(X, Sym^k \underline{L})$ is freely generated by $\{u_i\}_{i=1}^{k+1}$ plus some torsion part.  Also, we have  the following facts (see Corollary 4.10 of \cite{CHMSS}):
\begin{enumerate}
\item
$\{\partial u_i\}_{i=1}^k$ generates $\pi_1(Sym^k \underline{L})$.
\item
For $1 \le i \le k$, we have $u_i \cdot \Delta =0$.
\item
$u_i \cdot D_j =\delta_{ij}$. In particular, $u_i \cdot D_{k+1}=0$ for any $1\le i\le k$.
\end{enumerate}
Identify $\mathbb{S}^2 -U_{z_{k+1}}$ with    $\mathbb{D}$ topologically.  We claim that there exists  a disk $$u': (\mathbb{D}, \partial \mathbb{D}) \to (Sym^k \mathbb{D}, Sym^k \underline{L})$$ such that $\partial u' =\mathbf{p}_{+} \cup \partial u \cup (-\mathbf{p}_{-})$ and $u'$ is disjoint from $\Delta$.  Note that  the claim is  not a consequence of the Whitney trick because $\pi_1(X - \Delta) \to \pi_1(X) =0$ is not injective.  Suppose that
$[\mathbf{p}_{+} \cup \partial u \cup (-\mathbf{p}_{-})] = \Pi_{j} (\partial u_{i_j})^{\pm 1}$ in $\pi_1(Sym^k \underline{L})$, where $1\le i_j \le k$.  Here we may  have $i_j = i_{j'}$ for $j \ne j'. $  Let $\varphi_i: \mathbb{D} \to B_i$ be a diffeomorphism for $1\le i \le k$.  Define maps $v_{i_j}(z): =[\varphi_{i_j}(z), z^1_{i_j}, ..., z_{i_j}^{k-1}]$, where $z^1_{i_j}, ..., z_{i_j}^{k-1}$ are  fixed points in the  circles other than $L_{i_j}$.  Then $v_{i_j}: (\mathbb{D}, \partial \mathbb{D}) \to  (Sym^k \mathbb{D}, Sym^k \underline{L})$ is a disk representing  the class $u_{i_j}$. We also require that $[z^1_{i_j}, ...,z_{i_j}^{k-1}]$ are distinct for different $j$. As a result,  $v_{i_j}$ are distinct maps.    Concatenating these disks produces a disk 
\begin{equation*}
\Pi_j v^{\pm 1}_{i_j} :  (\mathbb{D}, \partial \mathbb{D}) \to  (Sym^k \mathbb{D}, Sym^k \underline{L}) 
\end{equation*}
such  that  $\partial ([\Pi_j v^{\pm 1}_{i_j}] )= [\mathbf{p}_{+} \cup \partial u \cup (-\mathbf{p}_{-})] $. Note that $v_{i_j}$ are disjoint from $\Delta$, so is $\Pi_j v^{\pm 1}_{i_j}$.  Take a homotopy $v'$ between   $\partial ( \Pi_j v^{\pm 1}_{i_j} ) $ and $\mathbf{p}_{+} \cup \partial u \cup (-\mathbf{p}_{-}) $ in $Sym^k \underline{L}$.  Gluing  $ \Pi_j v^{\pm 1}_{i_j}$ and $v'$ along   $\partial ( \Pi_j v^{\pm 1}_{i_j} ) $,  then the result $u'$ satisfies our requirements.

Finally, gluing $u$ and $u'$ along $\mathbf{p}_{+} \cup \partial u \cup (-\mathbf{p}_{-})$, then  we get a cylinder $\bar{u}: [- 1, 1]\times S^1 \to Sym^k \mathbb{D}$ (after reparametrization) such that $\bar{u}(\pm 1, t) = \bar{ \gamma}_{\pm}(t)$ and $\bar{u}$ is disjoint from $\Delta$. This gives a homotopy between $\bar{\gamma}_+$ and $\bar{\gamma}_-$. Therefore, $b(\varphi_+, \underline{L}) = b(\varphi_-, \underline{L}).$
\end{proof}

Shenzhen University

\href{mailto: ghchen@szu.edu.cn} {\tt ghchen@szu.edu.cn} \\
\end{document}